\documentclass[english,12pt]{amsart}
\usepackage{amsmath,amsfonts,bbm,amsthm}
\usepackage[latin1]{inputenc}  
\usepackage[T1]{fontenc}       
\usepackage[psamsfonts]{amssymb}

\usepackage{latexsym}
\usepackage[dvips]{epsfig}
\usepackage{dcolumn}
\usepackage{tabularx}
\usepackage{longtable}
\usepackage{graphics}
\usepackage{pstricks}
\usepackage{pst-node}
\usepackage{latexsym}
\usepackage{babel}
\usepackage{euscript}

\renewcommand{\theequation}{\thesection.\arabic{equation}}
\newtheorem{theorem}{Theorem}
\newtheorem{lemma}{Lemma}
\newtheorem{proposition}{Proposition}
\newtheorem{corollary}{Corollary}
\newtheorem{remark}{Remark}
\newtheorem{definition}{Definition}

\newcommand{\eqnsection}{
\renewcommand{\theequation}{\thesection.\arabic{equation}}
    \makeatletter
    \csname  @addtoreset\endcsname{equation}{section}
    \makeatother}
\eqnsection
\def\sERRW{$\star$-ERRW }
\def\sVRJP{$\star$-VRJP }

\def\w{\omega}

\def\sss{{\mathcal S}}
\def\demi{{1\over 2}}

\def\P{{{\Bbb P}}}
\def\N{{{\Bbb N}}}
\def\R{{{\Bbb R}}}

\def\ppp{{{\mathcal P}}}
\def\aaa{{{\mathcal A}}}
\def\lll{{{\mathcal L}}}

\def\im{{{\hbox{Im}}}}

\def\E{{{\Bbb E}}}
\def\dive{{\hbox{div}}}

\def\hhh{{{\mathcal H}}}

\def\indic{{{\mathbbm 1}}}
\def\Id{{\hbox{Id}}}

\newtheorem{corol}{Corollary}[section]
\newtheorem{rema}{Remark}[section]
\newtheorem{exa}{Example}

\renewcommand{\le}{\leqslant}

\renewcommand{\ge}{\geqslant}

\renewcommand{\subset}{\subseteq}

\newcommand{\bal}{\begin{align*}}
\newcommand{\eal}{\end{align*}}
\newcommand{\beq}{\begin{eqnarray*}}
\newcommand{\eeq}{\end{eqnarray*}}
\newcommand{\bte}{\begin{theorem}}
\newcommand{\ete}{\end{theorem}}
\newcommand{\bl}{\begin{lemma}}
\newcommand{\el}{\end{lemma}}
\newcommand{\bd}{\begin{description}}
\newcommand{\ed}{\end{description}}

\newcommand{\bc}{\begin{cases}}
\newcommand{\ec}{\end{cases}}
\newcommand{\bp}{\begin{proof}}
\newcommand{\ep}{\end{proof}}
\newcommand{\bco}{\begin{corol}}
\newcommand{\eco}{\end{corol}}

\newcommand{\iy}{\infty}
\newcommand{\tx}{\text}

\newcommand{\Pb}{\mathbb{P}}

\newcommand{\al}{\alpha}

\newcommand{\be}{\beta}

\newcommand{\g}{\gamma}

\newcommand{\G}{\Gamma}

\newcommand{\ze}{\zeta}

\newcommand{\om}{\omega}

\def\bdes{\begin{description}}
\def\edes{\end{description}}

\def\kkk{{\mathcal K}}

\def\quart{{1\over 4}}
\def\un{{\bf 1}}
\def\sERRW{{$\star$-ERRW }}
\def\sERRWse{{$\star$-ERRW}}
\def\ggg{{\mathcal G}}

\def\red{}
\def\blue{}

\begin{document}
\author[S. Bacallado]{Sergio Bacallado}
\author[C. Sabot]{Christophe Sabot}
\address{Statistical Laboratory, D.1.10
Center for the Mathematical Sciences,
Wilberforce, Rd. Cambridge, CB3 0WB, United Kingdom}\email{sb2116@cam.ac.uk}
\address{Universit\'e Claude Bernard Lyon 1,
Institut Camille Jordan, CNRS UMR 5208, 43, Boulevard du 11 novembre 1918,
69622 Villeurbanne Cedex, France and Institut Universitaire de France} \email{sabot@math.univ-lyon1.fr}
\author[P. Tarr\`es]{Pierre~Tarr\`es}
\address{NYU-ECNU Institute of Mathematical Sciences at NYU Shanghai, China; Courant Institute of Mathematical Sciences, New York, USA; CNRS and Universit\'e Paris-Sorbonne, Laboratoire de Probabilit\'es, Statistique et Mod\'elisation, 75005 Paris, France. }\email{tarres@nyu.edu}
\title{The $\star$-Edge-Reinforced Random Walk}
\maketitle
\begin{abstract}
We define 
a linearly reinforced process called the $*$-Edge-Reinforced Random Walk (\sERRW) which can be seen as a Yaglom reversible, hence non-reversible, extension of the Edge-Reinforced Random Walk (ERRW) introduced by Coppersmith and Diaconis in 1986 \cite{coppersmith}.
This family of processes also generalizes the r-dependent ERRW introduced by Bacallado in 2009 \cite{Bacallado1}.
Under some assumptions on the initial weights, the \sERRW is  partially exchangeable in the sense of Diaconis and Freedman \cite{diaconis-freedman}, and thus it is a random walk in a random environment. The main result of the paper gives the explicit expression of the mixing law, hence extending the "magic formula" of Coppersmith and Diaconis from the case of mixtures of reversible Markov chains to the case of mixtures of Yaglom reversible Markov chains.
\end{abstract}
\section{Introduction}
\subsection{$\star$-directed graphs}
In this paper, we consider a directed graph $\ggg=(V,E)$ endowed with an involution on the vertices denoted by $*$, and such that
\begin{equation}
\label{invol}
(i,j)\in E\;\;\; \Longleftrightarrow \;\;\;(j^*,i^*)\in E,
\end{equation}
hence $\star$ is also an involution on $E$ by $\star((i,j))=(i,j)^*=(j^*,i^*)$. In the sequel, we sometimes write $i\to j$ for $(i,j)\in E$.
Denote by $V_0$ the set of fixed points of $\star$ 
$$
V_0=\{i\in V \hbox{ s.t. } i=i^*\},
$$
and by $V_1$ a subset of $V$ such that $V$ is a disjoint union
$$
V= V_0\sqcup V_1\sqcup V_1^*.
$$

Denote by $\tilde E$ the set of edges quotiented by the relation $(i,j)\sim(j^*,i^*)$.
For simplicity of notations, we also consider $\tilde E$ as a subset of $E$ obtained
by choosing a representative among each pairs of equivalent edges $(i,j)$ and $(j^*,i^*)$.

{\red
{\blue In our context,} a Markov Chain {\blue on the graph $\ggg$} with transition probability $p(i,j)$ from $i$ to $j$
will be called Yaglom reversible if and only if there exists a probability measure $\pi$ on $V$ such that, for all $i$, $j$ $\in V$, $(i,j)\in E$,  
\begin{align*}
\pi(i)p(i,j)&=\pi(j^*)p(j^*,i^*)\\
\pi(i)&=\pi(i^*).
\end{align*}
Note that this implies in particular that $\pi$ is an invariant measure for the Markov Chain.
This notion of Yaglom reversibility can be seen as  a generalization of the classic notion of reversibility {\blue in the same spirit as the notion introduced by} introduced by Yaglom  in \cite{Yaglom0,Yaglom} (see also \cite{Dobrushin}) in a continuous time and space setting, where the involution $\star$ {\blue is the counterpart of} the map $(x,\dot{x})\mapsto(x,-\dot{x})$ on the space and velocity coordinates. {\blue As we will see in the sequel, Yaglom reversible Markov chains appear naturally as mixing laws of the \sERRW.}
}

\subsection{The $\star$-Edge Reinforced Random Walk ($\star$-ERRW)}
\label{serrw} 
Suppose we are given some positive weights $(\alpha_{i,j})_{(i,j)\in E}$ on the edges invariant by the involution, i.e. such that
$$
\alpha_{i,j}=\alpha_{j^*,i^*}\text{ for all }(i,j)\in E.
$$

Let $(X_n)_{n\in\N}$ be a nearest-neighbor discrete-time process taking values in $V$. For all $(i,j)\in E$, denote by $N_{i,j}(n)$
the number of crossings of the edge $(i,j)$, i.e. 
$$
N_{i,j}(n)=\sum_{k=0}^{n-1}\indic_{(X_k,X_{k+1})=(i,j)},
$$
and let  
\begin{eqnarray}\label{alphan}
\alpha_{i,j}(n)= \alpha_{i,j}+ N_{i,j}(n)+N_{j^*,i^*}(n).
\end{eqnarray}

The process $(X_n)_{n\in\N}$ is called a $\star$-Edge Reinforced Random Walk (\sERRW) starting from $i_0$ 
if $X_0=i_0$ and, for all $n\in\N$ and $j\in V$, 
$$
\P(X_{n+1}=j\;|\; X_0,X_1, \ldots X_n)= \indic_{X_n\rightarrow j}{\alpha_{X_n,j}(n)\over \sum_{l, X_n\to l} \alpha_{X_n,l}(n)}.
$$
{\red

{\blue The \sERRW includes, as special cases, several models of reinforced random walks which already appeared in the literature. 

Firstly,} the \sERRW is a generalisation of the Edge-Reinforced Random Walk (ERRW) introduced in the seminal work of Diaconis and Coppersmith \cite{coppersmith}, which itself corresponds to the case where $\star$ is the identity map.}

{\blue Secondly, the} Random Walk in random Dirichlet Environment (RWDE), considered first in \cite{Enriquez-Sabot06} and analyzed in a series of papers (see \cite{Sabot-Tournier17} for a review), can also be seen as a particular case of \sERRWse. 
Indeed, given a directed graph $\ggg_1=(V_1, E_1)$, consider its "reversed" graph $\ggg_1=(V_1^*, E_1^*)$, where $V_1^*$ is a copy of $V_1$ and $E_1^*$ is obtained by reversing each edge in $E_1$. Then the full graph $\ggg$ is a disjoint union $(V_1\sqcup V_1^*,E_1\sqcup E_1^*)$. If $i_0\in V_1$, then the \sERRW is a RWDE on $\ggg_1$. In this sense, the \sERRW can be seen as an interpolation between the ERRW and the RWDE. An interesting modification of that construction is obtained by gluing the two graphs $\ggg_1$ and $\ggg_1^*$ at the starting point $i_0$.{ \blue  For that modification, the condition for partial exchangeability of the \sERRW (Proposition~\ref{bac} below) is the one under which the key property of statistical invariance by time reversal for RWDE holds (see \cite{sabot1} Lemma 1, and Remark~\ref{RWDE} below). 

Thirdly, motivated by questions coming from Bayesian statistics described in Section \ref{sec_Bayesian}, Bacallado \cite{bacallado2,bacallado2016bayesian} (see Section \ref{sec_Bayesian}     below) introduced models of $k$-dependent edge reinforced random walk and of reinforced random walk with amnesia, which can be seen as a special case of the \sERRWse. For these last two models, partial exchangeability, and thus a representation of the reinforced walk as a mixture of Markov chains, were proved under a condition on the initial weights, but the mixing law was not computed. 

In this paper, we prove partial exchangeability of the \sERRW under a simple condition on the weights, and the main result, Theorem~\ref{main}, gives an explicit formula for the mixing law, which lives on the space of Yaglom reversible Markov chains of the graph $\ggg$. Hence it extends the formula of Diaconis and Coppersmith for the ERRW. In particular, it provides the mixing measure for the examples described above, but the context of the \sERRW is much larger than these special models. 

Similarly as for the ERRW, a $\star$-Vertex Reinforced Jump Process can be associated with the \sERRWse. This process, which is introduced in this paper and analyzed in two separate papers \cite{starvrjp1,starvrjp2}, has a remarkable structure and exhibits several new phenomena. Finally, note that Yaglom reversibility also appears naturally in different models of self-interacting processes, see \cite{Tarres_Toth_Valko,Horvath_Toth_Veto}.}
\subsection{
{\blue Motivation from Bayesian statistics}.}
\label{sec_Bayesian}
{\red 
Recall that the ERRW introduced by Diaconis and Coppersmith is not a Markov Chain, but that it satisfies a partial exchangeability property, which implies a de Finetti-type representation as a mixture of reversible Markov chains.
Diaconis and Rolles \cite{diaconis-rolles} noted that the mixing law  $\mathcal{L}(i_0,\al)$ on the weights of the reversible Markov chains  of the ERRW starting from $i_0$ and initial weights $(\al_{i,j})$ can be seen as a prior distribution for Bayesian analysis of reversible Markov chains. Indeed, the posterior distribution of those weights, conditionally on the first $n$ steps of the walk, is by definition given by   $\mathcal{L}(X_n,\al(n))$, and thus it is a conjugate prior.
This prior has several appealing features including (i) a formula for the predictive distribution conditional on an observation of the Markov chain, and for other posterior moments, such as cycle probabilities, (ii) full support of the mixing distribution on the space of reversible Markov kernels, and (iii) a characterisation as the only prior with certain sufficient statistics for the predictive distribution, which motivates the choice from a subjectivist perspective \cite{rolles2003edge}. 

Bacallado \cite{bacallado2} extended this idea by introducing the $k$-dependent Reinforced Random Walk, which can be seen as a $\star$-ERRW on the de Bruijn graph $\ggg=(V,E)$ of a finite set $S$, defined as follows: $V=S^k$, $E=\{((i_1, \ldots i_k),(i_2, \ldots, i_{k+1})),\,i_1,\ldots i_{k+1}\in S\}$, with $\star$ the involution which maps $(i_1, \ldots i_k)\in V$ to the reversed sequence $(i_k, \ldots i_1)\in V$, which obviously satisfies \eqref{invol}; thus $V_0$ is the set of palindromes.
They enable to introduce a prior distribution for higher-order Markov chains which are reversible, in the sense that $(X_n)_{n\in \mathbb Z}$ is identically distributed to the time reversal $(X^*_{-n})_{n\in\mathbb Z}$ under the invariant measure. This work was motivated by an application to the statistical analysis of molecular dynamics simulations with microscopically reversible laws. 

Bacallado also introduces in \cite{bacallado2} priors on the variable-order Markov chains, again through a particular case of \sERRW, which provide a useful, parsimonious alternative to a higher-order Markov model. Consider a random walk $(X_n)_{n\in\mathbb N}$ on the de Bruijn graph; we say it is of variable-order with context set $\mathcal C\subseteq S\cup S^2\cup\dots\cup S^k$, if for each $(i_1,\dots,i_\ell)\in\mathcal C$, the transition probabilities out of $x$ and $y$ are the same whenever $x$ and $y$ both end in $(i_1,\dots,i_\ell)$. Bacallado \cite{bacallado2} shows that if $(X_n)_{n\in\mathbb N}$ is reversible in the higher-order sense, then $x\in \mathcal C$ implies that $x^*\in \mathcal C$ and any $y$ which has $x$ as a prefix is also in $\mathcal C$. This fact is used to define a prior with full support on the space of variable-order, reversible Markov chains with a specific context set. The prior shares many of the computational advantages of the ERRW prior of Diaconis and Rolles \cite{diaconis-rolles}.

Another example of $\star$-ERRW is the reinforced random walk with amnesia of \cite{bacallado2016bayesian}, which defines a prior for reversible processes with long memory. A random walk with amnesia is a random walk on the graph $\ggg=(V,E)$ defined by $V=S\cup S^2\cup\dots S^k$ with edges from $(i_1,\dots,i_m)\in V$ leading to $(i_2,\dots,i_m)$, if $m>1$, as well as to $(i_1,\dots,i_m,j)$ for each $j\in V$, if $m<k$. The first class of edge ``forgets'' the most distant element in $S$ in the history, while the latter appends a new element of $S$ to the history. A $\star$-ERRW on this graph defines a mixture of random walks with amnesia, which can be used in a Bayesian analysis to learn a context-dependent, stochastic length of memory under the assumption of reversibility, avoiding the need to specify a context set a priori.
 }

\subsection{The $\star$-Vertex Reinforced Jump Process ($\star$-VRJP)} 
Similarly as for the ERRW, the \sERRW can be represented in terms of a $\star$-Vertex-Reinforced Jump Process ($\star$-VRJP) with independent gamma conductances. 

More precisely, suppose we are given some weights $(W_{i,j})_{(i,j)\in E}$ invariant by the involution, that is, such that 
$$
W_{i,j}=W_{j^*,i^*} \text{ for all }(i,j)\in E.
$$

If $(u_i)_{i\in V}\in \R^V$, then let $u^*=(u_{i^*})_{i\in V}\in \R^V$, and define $W^u\in \R^E$ by
$$
W^u_{i,j}= W_{i,j}e^{u_i+u_{j^*}}.
$$
Remark that $W^u_{i,j}= W^u_{j^*,i^*}$.

Consider the process $(X_t)_{t\ge 0}$ which, depending on the past at time $t$, jumps from $i$ to $j$ at a rate
$$
W^{T(t)}_{i,j}=W_{i,j}e^{T_i(t)+T_{j^*}(t)},
$$
where $T(t)$ is the local time of $X$ at time $t$
$$
T_i(t)=\int_0^t \indic_{X_u =i} du.
$$
We call this process the
$\star$-VRJP. 
We note by $\P^W_{i_0}$ the law of the \sVRJP starting from the point $i_0$. We always consider that the \sVRJP is defined on the canonical space ${\mathcal{D}}([0,\infty),V)$ of c\`adl\`ag functions from $[0,\infty)$ to  $V$ and $(X_t)_{t\ge 0}$ will denote the canonical process on ${\mathcal{D}}([0,\infty),V)$.
The \sERRW can be viewed as a generalization of the VRJP since when $\star$ is the identity, it matches the definition of the VRJP in exponential time scale, see \cite{sabot-tarres} Lemma~1.

\begin{lemma}\label{lem_ERRW-VRJP}
Let $(\al_{i,j})_{(i,j)\in E}$ be a set of positive weights on the edges such that $\al_{i,j}=\al_{j^*,i^*}$.
Let $(W_e)_{e\in \tilde E}$ be independent Gamma random variables with parameters $\al_{i,j}$. Let $(X_n)_{n\in\N}$ be
the discrete time process describing the successive jumps of $(X_t)_{t\ge 0}$ and let $\P_{i_0}(\cdot )= \E\left(\P^W_{i_0}(\cdot )\right) $ be the annealed
law obtained after taking expectation with respect to the r.v. $W$. Then, under the annealed law $\P_{i_0}$,
$(X_n)_{n\ge 0}$ has the law of the $\star$-ERRW defined in the first part.
\end{lemma}
\begin{proof}
The proof is similar to that of Lemma 4.7 in \cite{tarres3} and Theorem 1 in \cite{sabot-tarres}.
\end{proof}

The \sVRJP is not partially exchangeable in the sense of Diaconis and Freedman \cite{diaconis-freedman}, see \cite{starvrjp1}. This leads to several new difficulties and phenomena.
The \sVRJP is analysed in separate papers by Sabot and Tarr\`es \cite{starvrjp1,starvrjp2}: after a suitable randomization of the initial local times, the \sVRJP becomes partially exchangeable. After that randomization, the limiting measure is computed and a random Schr\"{o}dinger operator representation is obtained, extending \cite{sabot-tarres,STZ17} in the \sVRJP case. The non-randomized \sVRJP can also be written as a mixture of conditioned Markov jump processes, after a proper time change.


\section{Statement of the results}
\subsection{Partial exchangeability of the $\star$-ERRW under assumptions on initial weights $\al$ and limiting manifold.}
\label{exch}
We use the notation that, for all $e=(i,j)\in E$, then $\underline e=i$ and $\overline e=j$. Let $\delta$ denote the delta Dirac function.

{\red
For all $i_0\in V$, $\al=(\al_{i,j})_{(i,j)\in \tilde E}\in(0,\infty)^{\tilde E}$, let 
\begin{equation}\label{gamma}
\gamma(i_0,\al)= {{\prod_{(i,j)\in \tilde E} \Gamma(\al_{i,j})}\over {\left( \prod_{i\in V_0} \Gamma( \demi( \al_i+1-\delta_{i_0}(i))2^{\demi (\al_i-\delta_{i_0}(i))}\right) 
\left( \prod_{i\in V_1} \Gamma( \min(\al_i, \al_{i^*}))  \right)}}.
\end{equation}
}

The divergence operator is defined as the operator $\dive: \R^E\to \R^V$ defined for a function on the edges $(x_e)_{e\in E}$ by
$$
\dive(x)(i)= \sum_{e, \underline e=i} x_e-\sum_{e, \overline e=i} x_e.
$$

For all $y\in\R^{E}$ let, for all $n\in\N$ and $i\in V$, 
$$y_i=\sum_{j:\,i\to j}y_{i,j}.$$

\begin{proposition}
\label{bac}
Let $i_0\in V$ be the starting point of the \sERRW and suppose that the weights $(\alpha_e)_{e\in E}$ have the following property
\begin{eqnarray}\label{divergence}
\dive(\alpha)=\delta_{i_0^*}-\delta_{i_0}.
\end{eqnarray}
Then the \sERRW $(X_n)$ is partially exchangeable in the sense of \cite{diaconis-freedman}.

{\blue
More precisely, if $\sigma=(\sigma_0, \ldots, \sigma_n)$ is a directed path in $\ggg$, we set $\alpha_e(\sigma)=\al_e+\overline n_e(\sigma)$, where $\overline n_{(i,j)}(\sigma)$ is the total number of crossings of oriented edges $(i,j)$ and $(j^*,i^*)$ along the path $\sigma$. Then
\begin{align*}
&\Pb(X_0=\sigma_0,\ldots,X_n=\sigma_n)
=\frac{\g(\sigma_n,\alpha(\sigma))}{\g(\sigma_0,\al)}.
\end{align*}
}
\end{proposition}
\begin{remark}\label{RWDE}
Let us come back to the special case described in the introduction, where $\ggg$ is obtained from a directed graph $\ggg_1=(V_1,E_1)$ glued to $\ggg_1^*=(V_1^*,E_1^*)$  by the starting point $i_0$, where $V_1^*$ is a copy of $V_1$ and $E_1^*$ is the set of edges $(j^*,i^*)$ for $(i,j)\in E$. Then $i_0=i_0^*$ and the condition for exchangeability of the \sERRW is $\hbox{div}(\alpha)=0$, which is the same as the condition introduced in \cite{sabot1} for the key property of statistical invariance by time reversal for Random Walks in Dirichlet Environments (RWDE). In fact, for this special graph, the \sERRW makes loops based at $i_0$, either in the forward direction in $\ggg_1$ or in the backward direction in $\ggg_1^*$. The partial exchangeability essentially is a direct consequence of the obervation that forward/backward choices do not count in the probability of paths, similarly as in the proof of statistical invariance by time reversal of the RWDE (see proof of Proposition~1 in \cite{Sabot-Tournier10}).  
\end{remark}

Set
\begin{eqnarray}
\label{hhh}
&\hhh= \{ (x_e)\in \R^E, \; x_{i,j}=x_{j^*,i^*}\}\simeq \R^{\tilde E}, \\
\label{lll0}
&\lll_0=\{x=(x_e)_{e\in E} \in \hhh, \hbox{s.t. $\hbox{div}(x)=0$ and $\sum_{e\in \tilde E} x_e=0$}\},
\\
\label{lll1}
&\lll_1=\{x=(x_e)_{e\in E}\in \hhh\cap(0,\infty)^E, \hbox{ s.t. $\hbox{div}(x)=0$ and $\sum_{e\in \tilde E} x_e=1$}\}.
\end{eqnarray}
As shown in the lemma below, $\lll_1$ is the open set of possible limits of the asymptotic occupation measure of the edges, and later $\lll_0$ will appear as the tangent space of $\lll_1$.
\begin{corollary}\label{cor1}
Under assumption \eqref{divergence}, the limit
$$
\lim_{n\to \infty} {1\over  n} N_{i,j}(n)= Y_{i,j}
$$
exists a.s. and $(Y_{i,j})_{(i,j)\in E}\in\lll_1$.
\end{corollary}
\subsection{Mixing measure of the \sERRW}
{\red
In the sequel we will always suppose that the graph is strongly connected, i.e. that for any vertices $i$ and $j$, there exists a directed path in the graph between $i$ and $j$.
}

\begin{definition}[Reference measure on $\lll_0$ and $\lll_1$]
\label{lem_bases}
By a slight abuse of notation, we will call basis of the subset $\lll_0$ a maximal subset $B\subset \tilde E$
such that the family  $y_e$, $e\in B$ are independent variables (i.e. linear forms) in $\lll_0$. Note that $\lll_1$ is an affine space directed by $\lll_0$.  In the sequel we will use as reference measure on $\lll_0$ and $\lll_1$
$$
dy_{\lll_0} = \prod_{e\in B} dy_e,\,\,\,dy_{\lll_1} = \prod_{e\in B} dy_e
$$ 
the measures obtained from the basis $B$. In Appendix \ref{bases} we describe the bases of $\lll_0$ and show that
$dy_{\lll_0} $ is independent of the choice of basis $B$.
\end{definition}
\begin{remark}
By $ \prod_{e\in B} dy_e$ we mean the pull-back of the Lebesgue measure on $\R^B$ obtained by the projection $(y_e)\in \lll_0\mapsto (y_e)_{e\in B}$. The measure $dy_{\lll_0}$ differs, by a universal constant, from the volume measure on $\lll_0$ induced by the Euclidean metric on $\R^{\tilde E}$.
\end{remark}

\begin{theorem}
\label{main}
{\bf(i)} The random variable $Y$ has the following distribution on $\lll_1$
$$
\mu_{i_0}^\al(dy) = C {\red \gamma(i_0, \al)^{-1} }\sqrt{y_{i_0}} \left({\prod_{(i,j)\in \tilde E} y_{i,j}^{\al_{i,j}-1}\over 
\prod_{i\in V} y_i^{\demi \al_i}} \right){1\over \prod_{i\in V_0} \sqrt{y_i}} \sqrt{ D(y)} \,dy_{\lll_1},
$$
where $\gamma(i_0, \al)$ is defined in \eqref{gamma} and
\begin{align*}
C&= {2\over \sqrt{2\pi}^{\vert V_0\vert -1} \sqrt{2}^{\vert V_0\vert +\vert V_1\vert}},\,\,D(y)=\sum_{T} \prod_{(i,j)\in T} y_{i,j}.
\end{align*}
The last sum runs on spanning trees directed towards a root $j_0\in V$, and the value does not depend on the choice of the
root $j_0$.

{\bf (ii)} The process $X_n$ is a mixture of {\blue Yaglom reversible} Markov chains with transition probabilities $p_{i,j}={y_{i,j}\over y_i }$, and invariant measure 
$y_i$ where $(y_{i,j})$ is distributed according to $\mu_{i_0}^a$.
\end{theorem}
\begin{remark}
Note that, as in the (reversible) case of Edge-Reinforced Random Walk (ERRW), the property that the integral of the measure $\mu_{i_0}^\al(dy)$ in Theorem \ref{main} is $1$ is not easy to prove directly. Even the apparently simple fact that the measure is finite is not immediately obvious.
\end{remark}
\begin{remark}
The sum on spanning trees  $D(y)$ can also be expressed as a determinant : 
$$
D(y)= \det(M(y))_{V\setminus\{j_0\}}
$$
where $\det(M(y))_{V\setminus\{j_0\}}$ is the diagonal minor obtained after removing line and column $j_0$ of the matrix $M(y)$ defined by 
$$
M(y)_{i,j}=\left\{\begin{array}{ll} y_{i,j}, &\hbox{ if $i\neq j$}
\\
-y_i, &\hbox{ if $i=j$}
\end{array}\right.
$$
The value of this determinant does not depend on the choice of $j_0$ since the sums on any line or column of $M(y)$ are null.
\end{remark}
\begin{remark}
Compared to the "magic formula" for ERRW in Coppersmith, Diaconis, \cite{coppersmith,diaconis-rolles,keane-rolles1}, the determinantal term $\sqrt{D(y)}$ is the same but it is now non symmetric. 
Also the measure is restricted to the limiting subspace $\lll_1$.
\end{remark}
\begin{remark}
One can check that the measure indeed coincides with the magic formula in the case of classic ERRW (cf \cite{coppersmith,diaconis-rolles,keane-rolles1}). When the graph is the disconnected union $(V_1\sqcup V_1^*,E_1\sqcup E_1^*)$ described at the end of Section \ref{serrw} we can easily confirm that it
coincides with the measure on the occupation time of the edges of Random Walks in Random Dirichlet Environment computed in lemma 2 \cite{sabot1} (ArXiv version).
\end{remark}
\section{Proof of the results of section \ref{exch}}
\begin{proof}[Proof of Proposition~\ref{bac}.]
{\blue
For $k\le n$, we denote by $\overline n_{(i,j)}(\sigma, k)$ the number of crossings of crossings of oriented edges $(i,j)$ and $(j^*,i^*)$ along the path $\sigma$ up to time $k$. Following \eqref{alphan}, we set $\alpha_{e}(\sigma, k)=\alpha_e+\overline n_{(i,j)}(\sigma, k)$, so that $\alpha_{(i,j)}(\sigma, n)=\alpha_{(i,j)}(\sigma)$. We also write $\alpha_i(\sigma,k)=\sum_{j,i\to j}\alpha_{(i,j)}(\sigma, k)$. Then, by construction,
$$\Pb(X_0=\sigma_0,\ldots,X_n=\sigma_n)=
\frac{\prod\limits_{e\in \tilde{E}} \Big(\prod\limits_{1\le k\le n:\,(\sigma_{k-1},\sigma_k)=e,e^*} \al_e(\sigma,k)\Big)}{\prod\limits_{i\in V}\Big(\prod\limits_{0\le k\le n-1:\,\sigma_k=i}\al_i(\sigma,k)\Big)}.
$$
Now, for all $e\in \tilde{E}$, $\al_e(\sigma,k)$ increases at each crossing of $e$ or $e^*$, thus
$$\prod\limits_{1\le k\le n:\,(\sigma_{k-1},\sigma_k)=e,e^*}\al_e(\sigma,k)
=\al_e(\al_e+1)\ldots(\al_e+\overline n_e(\sigma)-1)=\frac{\G(\alpha_e(\sigma))}{\G(\al_e)}.$$
Besides, for all $i\in V_0$, $\al_i(\sigma,k)$ increases by two between two successive visits to $i$: indeed, the weight of an edge $(i,j)$ is increased at the exit of $i$, and the weight of an edge $(j,i)$ is increased at the return to $i$, hence increasing also the weight of the edge $(i,j^*)$ since $i^*=i$. Besides, $\alpha_i(\sigma,k)$ has value $\al_i+1-\delta_{i_0}(i)$ at its first visit to $i$ and value $\al_i(\sigma)-1-\delta_{\sigma_n}(i)$ at the last visit before time $n-1$. Thus, 
\begin{align*}
&\prod\limits_{i\in V_0}\Big(\prod\limits_{0\le k\le n-1:\,\sigma_k=i}\al_i(\sigma,k)\Big)\\
=&(\al_i+1-\delta_{i_0}(i)) ((\al_i+3-\delta_{i_0}(i))\ldots(\al_i(\sigma)-2+(1-\delta_{i_n}(i)))\\
=&\frac{\Gamma( \demi( \alpha_i(\sigma)+1-\delta_{i_n}(i))2^{\demi (\alpha_i(\sigma)-\delta_{i_n}(i))}}{\Gamma( \demi( \al_i+1-\delta_{i_0}(i))2^{\demi (\al_i-\delta_{i_0}(i))}}
\end{align*}
Finally, we can compute the terms in the denominator for $i\in V_1$ in a similar manner. Indeed, note that $\dive((\alpha_e(\sigma,k)))(i)=\alpha_i(\sigma,k)-\alpha_{i^*}(\sigma,k)$ 
for all vertex $i\in V$, and that, by induction, the property \eqref{divergence} holds for all $k\in\N$:
$$\dive((\alpha_e(\sigma,k)))=\delta_{\sigma_k^*}-\delta_{\sigma_k}.$$
It implies that $\min(\alpha_{\sigma_k}(\sigma,k),\alpha_{\sigma_k^*}(\sigma,k))$ is always reached at $\sigma_k$.
Besides $\min(\alpha_{i}(\sigma,k),\alpha_{i^*}(\sigma,k))$ increases by $1$ after each visit to $\{i,i^*\}$, so that
$$\prod\limits_{0\le k\le n-1:\,\sigma_k=i}\al_i(k)
=\frac{\Gamma( \min(\alpha_i(\sigma), \alpha_{i^*}(\sigma)))}{\Gamma( \min(\al_i, \al_{i^*}))}.$$
}
\end{proof}
\begin{proof}[Proof of Corollary~\ref{cor1}]
Since the process $X_{n}$ is partially exchangeable the limit above exists a.s.. Moreover, 
$$
\dive(N(n))=\delta_{i_0}-\delta_{i^*_0}-(\delta_{X_n} -\delta_{X^*_n} )
$$
which implies that $\dive(Y)=0$. Also $\sum_{(i,j)\in \tilde E} N_{i,j}(n)=n$, since at each time one edge of $\tilde E$ is reinforced. 

Let us now prove
that $Y_{i,j}=Y_{j^*,i^*}>0$. Remark that asymptotically $\al_{i,j}(n)/\al_i(n)\to (Y_{i,j}+Y_{j^*,i^*})/2Y_i$, where $Y_i=\sum_{i\to j} Y_{i,j}$, and note that $Y_i=Y_{i^*}$ since $\dive(Y)=0$. 

By a standard argument using Rubin's construction, $(X_n)$ visits a.s. all vertices of the graph. Besides, $X_n$ is a mixture of irreducible Markov chains by \cite{diaconis-freedman}, which implies that $Y_i>0$ for all $i\in V$, and  conditioned on $Y$, it is a Markov chain with transition probabilities $Y_{i,j}/Y_i$. Therefore $Y_{i,j}/Y_i= (Y_{i,j}+Y_{j^*,i^*})/2Y_i>0$,
which is equivalent to $Y_{i,j}=Y_{j^*, i^*}>0$.
\end{proof}

\section{Proof of Theorem \ref{main}}
\label{sec:proof}

{\red Recall the definition of $\hhh$ and  $\lll_1$ from \eqref{hhh} and \eqref{lll1}. Consider a test function $\varphi:\lll_1\mapsto \R_+$, $C^\infty$ with compact support in $\lll_1$. 
Let, 
for all initial weights
$(\alpha_{i,j})_{(i,j)\in {E}}\in\hhh\cap(0,\iy)^{{E}}$,  and $k\in V$, 
$$ \Psi(k,\al)(\varphi)=\int_{\lll_1} \varphi(y)\mu^{\al}_{k}(dy)
$$
Since $Y\in \lll_1$ a.s., in order to conclude the proof of Theorem~\ref{main} it is sufficient to prove that, for all such $\varphi$,
$(\alpha_{i,j})_{(i,j)\in E}$,
and $i_0\in V$, 
\begin{eqnarray}\label{eq_varphi}
\E_{i_0}^{(\al)}(\varphi(Y))= \Psi(i_0,\al)(\varphi).
\end{eqnarray} 
}
\begin{remark}
Note that, contrary to the Feynman-Kac approach for the VRJP in \cite{sabot-tarres}, we do not use Fourier transforms but integration against 
smooth compactly supported test functions, since even integrability of the measures $\mu^{\al}_{k}$ (which turn out to be probability measures) is not obvious a priori.
\end{remark}
{\red Fix now $\varphi$ and denote by $K\subset \lll_1$ its support. The rest of the section is devoted to the proof of  \eqref{eq_varphi}.
}

The process $(X_n, \al(n))$ is a Markov process on $V\times \R^{\tilde E}$ with generator
$$
Lf(i,\al)=\sum_{j, i\rightarrow j} \al_{i,j}\left(f(j,\al+\indic_{(i,j)}+\indic_{(j^*,i^*)})-f(i,\al)\right).
$$

Lemma \ref{fk} below states that $\Psi(X_n, \alpha(n))(\varphi)$ is a martingale for the Markov generator $L$. In the proof and later on, we will use the notation, given $i_0\in V$,  $y\in \lll_1$,
\begin{align*}
\eta_{\al}(y)&={\prod_{(i,j)\in \tilde E} y_{i,j}^{\al_{i,j}}\over 
\prod_{i\in V} y_i^{\demi \al_i}},\,\,\,\,\,\,
\rho(y)=\frac{\sqrt{y_{i_0}}   }{ \prod_{i\in V_0} \sqrt{y_i}\prod_{(i,j)\in \tilde E} y_{i,j}}\sqrt{ D(y)},\\
\ze_{i_0,\al}(y)&=C\gamma(i_0, \al)\eta_{\al}(y)\rho(y),
\end{align*}
so that 
$$
\mu_{i_0}^\al(dy) = \ze_{i_0,\al}(y)\,dy_{\lll_1}.
$$

\begin{lemma}(Feynman-Kac equation)
\label{fk}
We have
$$
L \Psi =0.
$$
\end{lemma}
\begin{proof}
When the process jumps from $i_0$ to $j_0$, then  the normalization term becomes
$$
\gamma(j_0,\al+\indic_{\{(i_0,j_0)\}}+\indic_{\{(j_0^*,i_0^*)\}})= {\al_{i_0}\over \al_{i_0,j_0}} \gamma(i_0,\al).
$$
Indeed, $\al_{i_0,j_0}$ is incremented by 1; when $i\in V_0$ the term  $\demi( \al_i+1-\indic_{i=i_0})$ is incremented by 1 if $i=i_0$,
whereas when $i\in V_1$ the term
$\inf(\al_i, \al_{i^*})$ is incremented by 1 if $i=i_0$ and $\inf(\al_{i_0},\al_{i^*_0})= \al_{i_0}$, since $\al_{i_0^*}=\al_{i_0}+1$. All the other terms are left unchanged. 

On the other hand, 
$$
\ze_{j_0,\al+\indic_{\{(i_0,j_0)\}}+\indic_{\{(j_0^*,i_0^*)\}}}(y)= {y_{i_0,j_0}\over y_{i_0}} \ze_{i_0,\al}(y).
$$
This immediately implies that $\forall i_0\in V$,
$$
\sum_{j_0, i_0\to j_0} \al_{i_0,j_0} \Psi(j_0,\al+\indic_{\{(i_0,j_0)\}}+\indic_{\{(j_0^*,i_0^*)\}})(\varphi) = \al_{i_0} \Psi(i_0,\al)(\varphi).
$$
\end{proof}
This implies that $\Psi(X_n, \al(n))(\varphi)$ is a martingale, and therefore that
$$
\Psi(i_0, \al)(\varphi)=\E_{i_0}^{(\al)}(\Psi(X_n, \al(n))(\varphi))
$$

 The next aim is to show that 
$$\lim_{n\to\iy} \Psi(X_n, \al(n))(\varphi) = \varphi(Y)\text{ a.s. },$$ where $Y=\lim_{n\to\iy}\al(n)/n$, which will imply \eqref{eq_varphi} by dominated convergence.

To that end we {\red show a domination bound and show that $\Psi(i_0,\al(n))(\varphi)$ approximates a Gaussian integral in Propositions~\ref{Domination}} and  \ref{Gaussian-limit}, proved in Section \ref{agauss}. Then we compute that Gaussian integral in Lemma \ref{Calcul-Q}, proved in Section \ref{comp}.

We use the notation that $a_n\sim_{n\to\iy} b_n$  iff $a_n/b_n\to_{n\to\iy}1$. 

For all $w\in (0,\iy)^E\cap\hhh$, we define the following bilinear form on $\hhh$ by
$$
Q_w(x,x')=\sum_{(i,j)\in E} {1\over w_{i,j}} x_{i,j} x'_{i,j} -\sum_{ i\in V} {1\over w_i} (Bx)_i(Bx')_i+(\sum_{(i,j)\in E} x_{i,j})(\sum_{(i,j)\in E} x'_{i,j}),
$$
where 
\begin{eqnarray}\label{B}
Bx_i= \demi(x_i+x_{i^*})=\demi( \sum_{j, i\to j} x_{i,j}+\sum_{j, j\to i} x_{j,i}). 
\end{eqnarray}
We will denote $Q_w$ by $Q$ when there is no confusion, and let $Q_w(x)=Q_w(x,x)$ by a slight abuse of notation. Remark that when $z\in \lll_0$, then $z_i=z_{i^*}$ and $B(z)(i)= z_i$ so that 
\begin{eqnarray}\label{Q}
Q_w(z)=2\sum_{(i,j)\in \tilde E} {1\over w_{i,j}} (z_{i,j})^2-\sum_{i\in V} {1\over w_i}(z_{i})^2
=\sum_{(i,j)\in  \tilde E}w_{i,j}\left(\frac{z_{i,j}}{w_{i,j}}-\frac{z_{i}}{w_{i}}\right)^2.
\end{eqnarray}

{\red \begin{proposition}\label{Domination}
We have the following uniform domination, $\P^{(\alpha)}_{i_0}$ a.s.,
$$
\Psi(X_n,\al(n))(\varphi)\le c,
$$
for a constant $c>0$ depending only on $i_0$, $\alpha$ and the function $\varphi$.
\end{proposition}}
 \begin{proposition}\label{Gaussian-limit}
We have a.s. 
 $$
\lim_{n\to{\infty}} \Psi(X_n,\al(n))(\varphi)
=
\varphi(Y) \sqrt{D(Y)} {C\sqrt{2\pi}^{\vert V_0\vert +\vert V_1\vert -\vert \tilde E \vert}\over
\left(\prod_{(i,j)\in \tilde E} \sqrt{ Y_{i,j}}\right)
\left( \prod_{i\in V_0\cup V_1}\sqrt{Y_i}\right)} \int_{\lll_0} e^{-{1\over 4}Q_{{\red Y}}(z)} (dz)_{\lll_0}.
 $$
\end{proposition}
\begin{lemma}\label{Calcul-Q}
We have for $Y\in \lll_1$ 
\beq
&&
\int_{\lll_0} e^{-{1\over 4}Q_Y(z)} (dz)_{\lll_0}
\\
&=&
\left( \prod_{(i,j)\in \tilde E} \sqrt{Y_{i,j}}\right)
\left( \prod_{i\in V_0\cup V_1} \sqrt{Y_i} \right) 
{\sqrt{ 2}^{\vert V_0\vert +\vert V_1\vert} \sqrt{ 2\pi}^{\vert \tilde E\vert -\vert V_1\vert -1}\over
 2 \sqrt{D(Y)}}
\eeq
\end{lemma}

\section{Proof of {\red Propositions~\ref{Domination}} and \ref{Gaussian-limit}}
\label{agauss}
Let, for all $n\in\N$, {\red 
$$\be(n)=\frac{\al(n)}{(\sum_{e\in \tilde E} \alpha_e) + n}$$
so that $\beta(n)\in \lll_1$ and $\beta(n)\to Y$ a.s.}
The following Lemma \ref{ag} analyses the asymptotic behavior  of $\gamma(X_n,\al(n))$.

In all the sequel, for simplicity of notations,  $c$ will denote a positive constant, which may change from a line to another, depending only on $i_0$, $\alpha$ and $\varphi$. 
\begin{lemma}
\label{ag}
We have, a.s.,
$$
\gamma(X_n,\al(n))
\sim
C \sqrt{n}^{\vert \tilde E\vert -\vert V_1\vert-1} \sqrt{2\pi}^{\vert V_0\vert+\vert V_1\vert -\vert \tilde E\vert}
 { \prod_{i,j\in \tilde E} \sqrt{Y_{i,j}} \over \sqrt{ Y_{X_n}}\prod_{i\in V_1} \sqrt{Y_i}}\left(\frac{1}{\eta_{\be(n)}(\be(n))}\right)^{(\sum_{\tilde E} \alpha_e)+n}. 
$$
{\red Besides, there exists a constant $c>0$, depending only on $\alpha$, such that a.s.
$$
\gamma(X_n,\al(n))
\le
c
\sqrt{2\pi}^{\vert V_0\vert+\vert V_1\vert -\vert \tilde E\vert}
 { \prod_{i,j\in \tilde E} \sqrt{\alpha_{i,j}(n)} \over \sqrt{ \alpha(n)_{X_n}}\prod_{i\in V_1} \sqrt{\alpha_i(n)}}\left(\frac{1}{\eta_{\be(n)}(\be(n))}\right)^{(\sum_{\tilde E} \alpha_e)+n}. 
$$
}
\end{lemma}
\begin{proof}
Recall the asymptotic behavior of the gamma function
$$
\Gamma(z) \sim_{z\to\iy} \sqrt{2\pi} z^{z-\demi} e^{-z},
$$
which in particular implies $ \Gamma(z\pm \demi) \sim z^{\pm \demi} \Gamma(z)$.

Therefore, for $i\in V_0$ and large $\al$, we have
\begin{eqnarray*}
\Gamma(\demi(\al_i+1-\indic_{\{i=x\}}))
&\sim& 
\sqrt{2\pi} \sqrt{\al_i\over 2}^{-\indic_{\{i=x\}}} ({\al_i\over 2})^{\demi \al_i} e^{-\demi \al_i}.
\end{eqnarray*}
Similarly, for $i\in V_1$,
\begin{eqnarray*}
\Gamma(\inf(\al_i, \al_{i^*}))=\Gamma\left(\demi(\al_i+\al_{i^*}-\indic_{\{i=x \; or \; i=x^*\}})\right)&\sim &{\sqrt{2\pi} \over \sqrt{\al_i}^{\indic_{\{i=x\; or \; x^*}\}}} {\al_i^{\demi \al_i} \al_{i^*}^{\demi \al_{i^*}}\over \sqrt{\al_i}}e^{-\demi (\al_i+\al_{i^*})}.
\end{eqnarray*}

This yields that for $n$ large that
\begin{eqnarray*}
\gamma(x,\al(n))
&\sim&
C \sqrt{2\pi}^{\vert V_0\vert+\vert V_1\vert -\vert \tilde E\vert}{1\over \sqrt{ \al_{x}(n)}}
\left(  {\prod_{i\in V} \al_i(n)^{\demi \al_i(n)}
\over 
\prod_{i,j\in \tilde E} \al_{i,j}(n)^{\al_{i,j}(n)}}\right)\left({ \prod_{i,j\in \tilde E} \sqrt{\al_{i,j}(n)} \over \prod_{i\in V_1} \sqrt{\al_i(n)}}\right)
\\
&=&
{\red C \sqrt{n}^{\vert \tilde E\vert -\vert V_1\vert-1} \sqrt{2\pi}^{\vert V_0\vert+\vert V_1\vert -\vert \tilde E\vert}
 { \prod_{i,j\in \tilde E} \sqrt{\beta_{i,j}(n)} \over \sqrt{ \beta_{X_n}(n)}\prod_{i\in V_1} \sqrt{\beta_i(n)}}{1\over \left(\eta_{\be(n)}(\be(n))\right)^{(\sum_{\tilde E} \alpha_e)+n}}.}
\end{eqnarray*}
where we use {\red in the first line} that $\sum \al_{i,j}(n)-\demi\sum \al_i(n) =0$. {\red This concludes the first part of the statement. The uniform lower bound is proved in a similar way, using the fact that $\alpha(n)_{i,j}\ge \alpha_{i,j}$ and that, for any $z_0>0$, we have a uniform control on the function $\Gamma(z)$ by 
$$
c_- \sqrt{2\pi} z^{z-\demi} e^{-z} \le \Gamma(z) \le c_+ \sqrt{2\pi} z^{z-\demi} e^{-z},
$$
for $z\ge z_0$.} 

\end{proof}

{\red Recall that $K\subset \lll_1$ is the compact support of $\varphi$. Since $\lll_1$ is an open set, there exists $\epsilon>0$ and a compact subset $K'\subset \lll_1$ such that $\hbox{dist}((K')^c,K)\ge \epsilon$. We fix such a compact subset $K'$ in the sequel.}  Let us now study the asymptotic behavior  of $\eta_{X_n, \al(n)}(y)$ as $n\to\iy$. We will need the following Lemma \ref{lemmax}.
\begin{lemma}
\label{lemmax}
Given $\be\in\lll_1$, the function $y\in \lll_1\to \log \eta_{\be}(y)$ is nonpositive and reaches its maximum on $\lll_1$ at $y= \be$, and 
\begin{align}
\label{taylor}
\log\left(\frac{\eta_\be(y)}{\eta_\be(\be)}\right)=-\frac{1}{4}Q_\be(y-\be)+O(\|y-\be\|_\iy^3), 
\end{align}
where {\red the estimate $O(\|y-\be\|_\iy^3)$ is uniform for $y\in K$ and $\beta$ in $K'$.
  Besides, there is a positive constant $c>0$ such that
$$
\log\left(\frac{\eta_\be(y)}{\eta_\be(\be)}\right)\le -c,
$$
for all $y\in K$ and $\beta \notin K'$.}
\end{lemma}
\begin{proof}
Let 
\begin{equation}
\label{py}
p_{ij}=\frac{y_{ij}}{y_i},\,\,\,p_{ij}^\be=\frac{\be_{ij}}{\be_i}.
\end{equation} We have 
\begin{equation}
\label{feta}
\log\eta_\be(y)=\frac{1}{2}\sum_{(i,j)\in E}\be_{i,j}\log(p_{i,j})\le0,
\end{equation}
thus
\begin{align}
\nonumber
\log\left(\frac{\eta_\be(y)}{\eta_\be(\be)}\right)&=\demi\sum_{(i,j)\in E}\be_{i,j}\log\left(\frac{p_{i,j}}{p_{i,j}^\be}\right)
=\demi\sum_{(i,j)\in E}\be_{i,j}\log\left(1+\frac{p_{i,j}-p_{i,j}^\be}{p_{i,j}^\be}\right)\\
\label{ete}
&\le\demi\sum_{(i,j)\in E}\be_i(p_{i,j}-p_{i,j}^\be)=0,
\end{align}
where we use in the inequality that $\log(1+x)\le x$ for all $x>-1$. 

Let $z=y-\be \in \lll_0$: a second order Taylor expansion of \eqref{ete} yields
\begin{align*}
\log\left(\frac{\eta_\be(y)}{\eta_\be(\be)}\right)&=-\frac{1}{4}\sum_{(i,j)\in E}\frac{\be_i^2}{\be_{i,j}}(p_{i,j}-p_{i,j}^\be)^2+O(\|y-\be\|_\iy^3)
\end{align*}
with a uniform control on the reminder term for $\beta\in K'$ and $y\in K$. Now note that
\begin{align*}
\sum_{(i,j)\in E}\frac{\be_i^2}{\be_{i,j}}(p_{i,j}-p_{i,j}^\be)^2
=\sum_{(i,j)\in E}\be_{i,j}\left(\frac{z_{i,j}/\be_{i,j}+1}{z_{i}/\be_i+1}-1\right)^2
=\sum_{(i,j)\in E}\frac{\be_{i,j}}{(z_{i}/\be_i+1)^2}\left(\frac{z_{i,j}}{\be_{i,j}}-\frac{z_{i}}{\be_{i}}\right)^2,
\end{align*}
which enables us to deduce \eqref{taylor} by \eqref{Q}.

{\red The function $(\beta,y)\to \log\left(\frac{\eta_\be(y)}{\eta_\be(\be)}\right)$ can be extended to $\overline \lll_1\times K$ by continuity by
$$
 \log\left(\frac{\eta_\be(y)}{\eta_\be(\be)}\right)= \demi\sum_{(i,j)\in E, \; \beta_{i,j}>0}\be_{i,j}\log\left(\frac{p_{i,j}}{p_{i,j}^\be}\right).
$$
Since $\dive(\beta)=0$ and $\sum_{i,j} \beta_{i,j}=1$, we have that $\log\left(\frac{\eta_\be(y)}{\eta_\be(\be)}\right)$ is negative for $\beta\in \overline \lll_1$ and $y\in K$. This implies that
$$
\sup_{\beta \in \lll_1\setminus K', \; y\in K} \log\left(\frac{\eta_\be(y)}{\eta_\be(\be)}\right) <0.
$$
 }
\end{proof}
{\red Now we first  bound $\Psi(X_n,\al(n))(\varphi)$ when $\beta(n)$ outside $K'$.
\begin{lemma}\label{bound_out}
There exists  a constant $c_2>0$ such that for all integer $n$, if $\beta(n)\notin K'$, a.s.
$$
\Psi(X_n,\al(n))(\varphi)\le \exp(-c_2 n).
$$
\end{lemma}
\begin{proof}
By Lemma~\ref{ag}, we have
\begin{eqnarray*}
&&\Psi(X_n,\al(n))(\varphi)
\\
&\le&  \sqrt{2\pi}^{\vert V_0\vert+\vert V_1\vert -\vert \tilde E\vert}
 { \prod_{i,j\in \tilde E} \sqrt{\alpha(n)_{i,j}} \over \sqrt{ \alpha(n)_{X_n}}\prod_{i\in V_1} \sqrt{\alpha(n)_i}}
 \int_K \left(\frac{\eta_{\be(n)}(y)}{\eta_{\be(n)}(\be(n))}\right)^{(\sum_{\tilde E} \alpha_e)+n} \varphi(y) \rho(y) dy_{\lll_1}
 \\
 &\le &
e^{-c_1 n}  { \prod_{i,j\in \tilde E} \sqrt{\alpha(n)_{i,j}} \over \sqrt{ \alpha(n)_{X_n}}\prod_{i\in V_1} \sqrt{\alpha(n)_i}} \int_K  \varphi(y) \rho(y) dy_{\lll_1}
\end{eqnarray*}
where we used the second part of Lemma~\ref{Q} to bound $\left(\frac{\eta_{\be(n)}(y)}{\eta_{\be(n)}(\be(n))}\right)^{(\sum_{\tilde E} \alpha_e)+n}$ since $\beta(n)\notin K'$ and $y\in K$. This enables us to conclude since $\alpha(n)$ is at most linear in $n$.
\end{proof}
}

Given $\eta\in(0,1/2)$, let
$$
K_n=\{y\in \lll_1 :\;\; \| y-\be(n)\|_\iy \le n^{-1/2+\eta}\}.
$$

\begin{lemma}
\label{neg}
For all $\eta\in(0,1/2)$, 
{\red there exists $c>0$ such that for $n$ large enough, if $\beta(n)\in K'$, a.s.

$$
\int_{K_{n}^c} \varphi(y) \mu_{X_n}^{\al(n)} (dy) \le \exp(-c n^{2\eta}).
$$ 
}\end{lemma}
\begin{proof}
Note that, if $\xi$ is the differentiable map defined on $\lll_1$ by $\xi(y)=p$, where $p$ given as in \eqref{py}, then $\xi$ is invertible and $\xi^{-1}$ is differentiable on $\xi(\lll_1)$. Indeed,  given $p\in\xi(\lll_1)$, there exists a unique $y=(y_{ij})_{(ij)\in E}\in\lll_1$ such that $\xi(y)=p$, defined as the invariant measure on $E$ of the Markov Chain with jump probability $p_{ij}$ from $i$ to $j$, and which is given by the Kirchhoff formula. {\red For simplicity, we denote $p^y$ for $\xi(y)$. As a consequence of the previous remarks, there exists a constant $c>0$ such that for any $\beta$ in the compact $K'$ and any $y$ in the compact $K$, we have $\|p^y-p^\be\|_\iy\l \ge c \|y-\be\|_\iy$. This implies in particular that if $\beta(n)\in K'$ and $y\in K_n^c\cap K$, then $ \|p^y-p^{\be(n)}\|_\iy> c n^{-1/2+\eta}$. 

Then, using that the function $p\to \demi\sum_{(i,j)\in E}\be_{i,j}(n)\log\left(\frac{p_{i,j}}{p_{i,j}^{\be(n)}}\right)$ is concave we deduce that, for $\beta(n)\in K'$,
\begin{eqnarray*}
\sup_{y\in K_{n}^c\cap K}\log\left(\frac{\eta_{\be(n)}(y)}{\eta_{\be(n)}(\be(n))}\right)
&=&\sup_{y\in K_{n}^c\cap K}\demi\sum_{(i,j)\in E}\be_{i,j}(n)\log\left(\frac{p^y_{i,j}}{p_{i,j}^{\be(n)}}\right)\\
&\le&
\sup_{y\in K, \; \|p^y-p^{\be(n)}\|_\iy> c_4 n^{-1/2+\eta}} \demi\sum_{(i,j)\in E}\be_{i,j}(n)\log\left(\frac{p^y_{i,j}}{p_{i,j}^{\be(n)}}\right)\\
&\le&\sup_{p\in \xi(\lll_1), \;\|p-p^{\be(n)}\|_\iy = c n^{-1/2+\eta}}\demi\sum_{(i,j)\in E}\be_{i,j}(n)\log\left(\frac{p_{i,j}}{p_{i,j}^{\be(n)}}\right)\\
&\le&-c n^{2\eta-1}.
\end{eqnarray*}
where we used concavity in the third line and Lemma~\ref{lemmax} in the fourth line.
}

On the other hand, by Lemma \ref{ag}, {\red for $\beta(n)\in K'$,}
$$\eta_{\al(n)}(\be(n))\g(X_n,\al(n))=(\eta_{\be(n)}(\be(n)))^{(\sum_{\tilde E} \alpha_e)+n}\g(X_n,\al(n))\le c n^{\frac{|\tilde{E}-|V_1|-1}{2}},$$
{\red for a constant $c>0$.} 
This implies that there is a constant $c>0$ such that for $n$ large enough
$$|\int_{K_{n}^c} \varphi(y) \mu_{X_n}^{\al(n)} (dy)|\le n^{\frac{|\tilde{E}-|V_1|-1}{2}}\exp(-c n^{2\eta})\int_{K_{n}^c} (\varphi\rho)(y)dy_{\lll_1}.$$
\end{proof}
{\red It remains now to prove that $\int_{K_{n}} \varphi(y) \mu_{X_n}^{\al(n)} (dy)$ has the good asymptotic behavior given in Proposition~\ref{Gaussian-limit} and that we have a uniform control on this integral for $\beta(n)\in K'$. } For all $y\in K_n$, by Lemmas \ref{ag} and \ref{lemmax},
\begin{align*}
&\ze_{X_n,\al(n)}(y)=C\gamma(X_n, \al(n))(\eta_{\be(n)}(\be(n))^{(\sum_{\tilde E} \alpha_e)+n}\frac{\eta_{\be(n)}(y)}{\eta_{\be(n)}(\be(n))}\frac{\sqrt{y_{X_n}}   }{ \prod_{i\in V_0} \sqrt{y_i}\prod_{(i,j)\in \tilde E} y_{i,j}}\sqrt{ D(y)}\\
&\sim C\left(\frac{\eta_{\be(n)}(y)}{\eta_{\be(n)}(\be(n))}\right)^{(\sum_{\tilde E} \alpha_e)+n}\frac{\sqrt{ D(Y)}}{ \prod_{(i,j)\in \tilde E} \sqrt{Y_{i,j}}\prod_{i\in V_0\cup V_1} \sqrt{Y_i}}
\sqrt{n}^{\vert \tilde E\vert -\vert V_1\vert-1} \sqrt{2\pi}^{\vert V_0\vert+\vert V_1\vert -\vert \tilde E\vert}\\
&\sim C\exp\left(-\frac{1}{4}Q_{\be(n)}(\sqrt{n}(y-\be(n)))\right)\frac{\sqrt{n}^{\vert \tilde E\vert -\vert V_1\vert-1} \sqrt{2\pi}^{\vert V_0\vert+\vert V_1\vert -\vert \tilde E\vert}\sqrt{ D(Y)}}{  \prod_{(i,j)\in \tilde E} \sqrt{Y_{i,j}}\prod_{i\in V_0\cup V_1} \sqrt{Y_i}}
\end{align*}
where the notation $a_n(y)\sim b_n(y)$ means here that $a_n(y)/b_n(y)$ is close to $1$ when $n$ is large and $y\in K_n$. 
By the change of variables $z=\sqrt{n}(y-\be(n))$, it yields a Jacobian factor $\sqrt{n}^{-\vert \tilde E\vert +\vert V_1\vert+1}$ (since $\lll_1$ has dimension $\vert \tilde E\vert -\vert V_1\vert-1$), which cancels the corresponding term above. 

{\red This implies, combined with Lemma~\ref{bound_out} and Lemma~\ref{neg}, that $\int_{K_{n}} \varphi(y) \mu_{X_n}^{\al(n)} (dy)$ has the asymptotic behavior given in Proposition~\ref{Gaussian-limit}.  Besides, for $\beta(n)\in K'$ and $y\in K\cap K_n$, Lemma~\ref{lemmax} gives a uniform control on the asymptotics of $\frac{\eta_{\be(n)}(y)}{\eta_{\be(n)}(\be(n))}$, hence yielding a uniform domination on  $\int_{K_{n}} \varphi(y) \mu_{X_n}^{\al(n)} (dy)$ for $\beta(n)\in K'$ which, again together with Lemma~\ref{bound_out} and Lemma~\ref{neg}, enables us to deduce Proposition~\ref{Domination}.}
\section{Computation of the Gaussian integral: Proof of Lemma \ref{Calcul-Q}}
\label{comp}

{\red For simplicity we simply write $Q$ for $Q_Y$ (and we consider here $Y$ as a determinist vector in $\lll_1$).} Let $D_{\tilde E}$ (resp. $D_V$) denote the diagonal $\tilde E\times \tilde E$ (resp. $V\times V$) matrix with diagonal coefficients $Y_{i,j}$
(resp. $Y_i$).
We have
\begin{eqnarray}
\nonumber
Q(x,x')&=& 2 \sum_{(i,j)\in \tilde E} {1\over Y_{i,j}} x_{i,j} x'_{i,j} -\sum_{i\in V} {1\over Y_i} Bx_iBx'_i+ (\sum_{i\in V} Bx_i)(\sum_{i\in V}Bx'_i) 
\\
&=&
\label{mform}
2  {}^tx( D_{\tilde E}^{-1}(\Id_{\tilde E}-\demi D_{\tilde E} {}^tBD_V^{-1}B+\demi D_{\tilde E} {}^t B\un {}^t \un B)) x'
\end{eqnarray}
where in the last formula $\un$ is the column vector of size $V$ and containing only 1's, and 
$B$ stands for the $V\times \tilde E $ matrix with coefficients
$$
B_{i,e}=\left\{\begin{array}{ll} \demi,&\hbox{ if $\underline e=i$ or $\underline e=i^*$ or $\overline e =i$ or $\overline e= i^*$}
\\
0,& \hbox{ otherwise}
\end{array}
\right.
$$
which is the matrix of the operator $B$ defined in (\ref{B}) when $\hhh$ is identified with $\R^{\tilde E}$. (N.B. : in the last formula we
understand the subset $\tilde E$ as a subset of $E$ obtained by taking a representative among two equivalent edges $(i,j)$ and $(j^*,i^*)$.) 

We denote by $G^Y$ the infinitesimal generator of the Markov Jump Process with rate $Y$
$$
G^Y f (i)= \sum_{j, i\rightarrow j} Y_{i,j}(f(j)-f(i)), \;\;\; \forall f\in \R^V.
$$
Since $Y$ is in $\lll_1$, $1$ is an invariant measure of $G^Y$ (i.e., $1 G^Y= 0$), hence $G^Y$ leaves invariant $\kkk_0=\{x\in\R^V:\,\sum_{i\in V}x_i=0\}$ so that $\im(G^Y)=\kkk_0$ by Perron-Frobenius theorem (existence of spectral gap), and $(G^Y)^{-1}$ is well-defined on $\kkk_0$.
 
Let 
$$
\sss=\{(h_i)_{i\in V}, \;\;\; h_i= h_{i^*}\} \simeq \R^{\tilde V},\,\,\aaa=\{(h_i)_{i\in V}, \;\;\; h_i= - h_{i^*}\} \simeq \R^{V_1}.
$$
and let us consider the corresponding orthogonal projections on $\sss$ and $\aaa$ given  by
$$
\ppp_\sss h _i =\demi(h_i+ h_{i^*}),\,\,\,\ppp_\aaa h _i =\demi(h_i- h_{i^*})
$$

\begin{lemma}\label{Ortho}
(Orthogonal decomposition) For all $x\in\hhh$, we have 
$$x=\lambda Y + \omega +z,$$
where $\lambda Y$, $\omega$ and  $z$ are orthogonal vectors for the $Q$-scalar product, $\dive(z)=0$, and $\lambda$, $\omega$ $\in\hhh_0$ are defined by
$$
\lambda = \sum_{(i,j)\in E} x_{i,j}\;\;\;\w_{i,j}= Y_{i,j} (v_{i^*}+v_{j})\text{ where }v= (G^Y)^{-1} (h)\in\aaa\text{ with }h= \demi \dive(x)\in \aaa.
$$
Note that
$$
Q(\w, \w)= -2<h, (G^Y)^{-1}h>.
$$
On the other hand we prove in Section \ref{jac} that the Jacobian of the change of variables 
$$
(x_e)\in \hhh\rightarrow (\lambda,(z_e),(h_i))\in \R\times\lll_0\times \aaa
$$ 
for the measures $dx_\hhh=\prod_{e\in \tilde E} dx_e$ on $\hhh$, $d z_{\lll_0}$ defined in Definition~\ref{lem_bases} and $dh_\aaa=\prod_{i\in V_1} dh_i$ on $\aaa$, is $\pm 1$, so that
\begin{eqnarray}\label{Gaussian}
\int_{\hhh} e^{-{1\over 4} Q(x)} dx_\hhh = 2\sqrt\pi  \int_{\lll_0} e^{-\quart Q(z)} dz_{\lll_0} \int_{\aaa} e^{\demi<h,(G^Y)^{-1}h>} \prod_{i\in V_1} dh_i.
\end{eqnarray}
\end{lemma}
\begin{proof}
Note that $x':=x-\lambda Y\in \hhh_0$ since $Y\in \lll_1$. Moreover for any $x\in \hhh$, we have
$$
Q(x,Y)= \sum_{i\to j} x_{i,j}-\sum_{i\in V} x_i+\left(\sum_{i\to j} Y_{i,j}\right)\left(\sum_{i\to j} x_{i,j}\right)= \sum_{i\to j} x_{i,j},
$$
 which implies $Q(Y,Y)=1$, $Q(z,Y)=Q(\w,Y)=0$ since $z$, $\om$ $\in \hhh_0$.
 
Let us denote by $q$ the operator on $\hhh_0$ such that $Q(y,y')=\left<y,qy'\right>$ for $y,y'\in \hhh_0$. Let us now find $\om\in\hhh_0$ such that $q\om=\nabla v$ for some $v$, and $\dive(\om)=\dive(x')$: note that this will directly imply that $\om$ is orthogonal, with respect to the scalar product $Q(.,.)$, to all $z\in\hhh_0$ such that $\dive(z)=0$ (and in particular to $z=x'-\om$), since $(\im \nabla)^\perp =\ker\dive$. Now it follows from \eqref{mform} that $q$ is a diagonal matrix with coefficients $D_{\tilde E}^{-1}$ on the subspace $\hhh\cap\{x\in\R^{\tilde{E}}\tx{ s.t. }x_{i^*}=-x_i\}$, on which we have
$(q\om)_{ij}=Y_{ij}^{-1}\om_{ij}$. Our first condition yields $\om_{ij}=Y_{ij}(v_j-v_i)$, and the second is satisfied in particular if
$\om_i=(G^Yv)_i=\dive(x)_i/2=h_i$ for all $i\in V$, which is equivalent to $v=(G^Y)^{-1}h$.
\end{proof}

\begin{lemma}
We have
\begin{eqnarray*}
&&\int_{\hhh} e^{-{1\over 4} Q(x)} dx 
\\
&=& \sqrt{2}^{\vert\tilde E\vert+\vert V_0\vert +\vert V_1\vert}\sqrt{\pi}^{\vert\tilde E\vert}\left(\prod_{(i,j)\in \tilde E} \sqrt{Y_{i,j}}\right)\left(\prod_{i\in  V_0\cup V_1} \sqrt{Y_i}\right)\det((-G^Y+ 2D_V\un {}^t \un D_V)_{|\sss\times\sss})
\end{eqnarray*}
\end{lemma}
\begin{proof}

It follows from \eqref{mform} that
$$
\int_{\hhh}  e^{-\quart Q(x)} dx = \sqrt{2\pi}^{\vert \tilde E\vert} {\left( \prod_{(i,j)\in \tilde E} \sqrt{Y_{i,j}}\right) \over\sqrt{ \det(\Id_{\tilde E}-\demi D_{\tilde E} {}^tBD_V^{-1}B+\demi D_{\tilde E} {}^t B\un {}^t\un B)}}
$$

Using the classical formula $\det( \Id_n + AB)=\det(\Id_m +BA)$ for $A$ (resp. $B$) a $n\times m$ (resp. $m\times n$) matrix, we deduce
\beq
\det\left(\Id_{\tilde E}-\demi D_{\tilde E} {}^tBD_V^{-1}B+\demi D_{\tilde E} {}^t B \un {}^t\un B\right)
&=&
\det\left(\Id_{\tilde E}-\demi D_{\tilde E} {}^tBD_V^{-1}+\demi D_{\tilde E} {}^t B\un {}^t \un B\right)
\\
&=&
\det \left(\Id_{V} -\demi B D_{\tilde E} {}^t BD_V^{-1}+ \demi B D_{\tilde E} {}^t B \un {}^t\un\right)
\eeq

Direct computation yields,  for $h\in \R^V$, $(i,j)\in \tilde E$
\beq
(D_{\tilde E} {}^t B h)_{i,j}= \demi Y_{i,j} (h_i+h_j+h_{i^*}+h_{j^*}),
\eeq
hence,
\beq
B D_{\tilde E} {}^t B h_i &=&{1\over 4}\left(\sum_{j, i\rightarrow j} Y_{i,j}(h_i+h_j+h_{i^*} +h_{j^*}) + \sum_{j,i^*\rightarrow j^*} Y_{i^*,j^*}(h_j+h_i+h_{j^*} +h_{i^*})\right)
\\
&=&
{1\over 4} (G^Yh_i+G^Yh^*_{i}+ G^Yh_{i^*}+G^Yh^*_{i^*})+ (Y_i h_i+Y_i h_{i^*})
\\
&=&
\ppp_\sss G^Y \ppp_\sss h_i + 2D_V \ppp_\sss h_i. 
\eeq
where in the second line $h^*$ is the vector $h^*_i= h_{i^*}$. Since $Y_i=Y_{i^*}$ we have 
$$D_V\ppp_\sss=\ppp_\sss D_V=\ppp_\sss  D_V\ppp_\sss,
$$ 
hence
$$
B D_{\tilde E} {}^t B= \ppp_\sss G^Y \ppp_\sss+2\ppp_\sss D_V \ppp_\sss.
$$
Similarly 
$$
B D_{\tilde E} {}^t B \un {}^t\un h_i = 2Y_i (\sum_{j\in V} h_j)
$$
thus $B D_{\tilde E} {}^t B \un {}^t\un = 2 D_V \un  {}^t \un = 2 \ppp_\sss D_V \un {}^t\un \ppp_\sss$.

In summary, we have
\beq
&&\det \left(\Id_{V} -\demi B D_{\tilde E} {}^t BD_V^{-1}+ \demi B D_{\tilde E} {}^t B \un {}^t\un\right)
\\
&=&
\det \left(\Id_{V}- \demi \ppp_\sss (G^Y(D_V)^{-1} +2\Id)\ppp_\sss+ \ppp_\sss D_V \un {}^t\un \ppp_\sss\right)
\\
&=&
\det\left( (-\demi  G^Y (D_V)^{-1}+ D_V \un {}^t \un)_{|\sss}\right)
\\
&=&
\left(\prod_{i\in  V_0\cup V_1} Y_i\right)^{-1} \det\left( (-\demi G^Y+ D_V \un {}^t \un D_V)_{|\sss}\right)
\\
&=&
2^{-\vert V_0\vert -\vert V_1\vert} 
\left(\prod_{i\in  V_0\cup V_1} Y_i\right)^{-1} \det\left( (- G^Y+ 2 D_V \un {}^t \un D_V)_{|\sss}\right)
\eeq
where by $\det(M_{|\sss})$ we understand the determinant of the matrix $\ppp_\sss M\ppp_\sss$ viewed as 
a linear operator from $\sss$ to $\sss$.
\end{proof}

Equation \eqref{Gaussian} now enables us to conclude. Indeed,
$$
\int_{\aaa} e^{\demi<h,(G^Y)^{-1}h>} \prod_{i\in V_1} dh_i= \det\left((-G_Y)^{-1}|_{\aaa}\right)=\det\left((-G^Y+ 2 D_V \un {}^t \un D_V)^{-1}|_{\aaa}\right),
$$
where we use in the last equality that $\ppp_\aaa (-G_Y)^{-1}\ppp_{\aaa}= \ppp_\aaa (-G^Y+ 2 D_V \un {}^t \un D_V)^{-1}\ppp_\aaa $ on $\aaa$.
 
On the other hand, 
\beq
&&\int_{\lll_0}  e^{-\quart Q(z)} (dz)_{\lll_0}
=
{1\over 2\sqrt{\pi} }{\int_{\hhh}  e^{-\quart Q(x)} dx 
  \over \int_{\aaa} e^{\demi<h,(G^Y)^{-1}h>} \prod_{i\in V_1} dh_i}
  \\
&=&
 \sqrt{2}^{\vert\tilde E\vert +\vert V_0\vert+\vert V_1\vert -2} \sqrt{\pi}^{\vert\tilde E\vert -\vert V_1\vert-1} 
 \left( \prod_{(i,j)\in \tilde E} \sqrt{Y_{i,j}}\right)\left( \prod_{i\in V_0\cup V_1} \sqrt{Y_i} \right)
\\
&&\boldsymbol{\cdot}\; {\sqrt{\det\left(((- G^Y+ 2 D_V \un {}^t \un D_V)^{-1})_{|\aaa}\right)}\over \sqrt{\det\left( (- G^Y+ 2 D_V \un {}^t \un  D_V)_{|\sss}\right)}}
\\
&=&
 \left( \prod_{(i,j)\in \tilde E} \sqrt{Y_{i,j}}\right)\left( \prod_{i\in V_0\cup V_1} \sqrt{Y_i} \right)
 { \sqrt{2\pi}^{\vert\tilde E\vert -\vert V_1\vert -1} \sqrt{2}^{\vert V_0\vert + \vert V_1\vert}\over \sqrt{2}\sqrt{\det(-G^Y+ 2 D_V \un {}^t \un D_V)}}
\eeq
since $\sss$ and $\aaa$ are orthogonal. In the last equality, we used that for a matrix $M$ with block decomposition on 
a linear decomposition of the space $E=E_1 + E_2$
$$
M
=\left(\begin{array}{ll} 
A &B
\\
C& D
\end{array}\right)
$$
then $ \det M = {\det(A)\over \det((M^{-1})_{| E_2})}$.
Finally, since the sum of lines (resp. columns) of $G^Y$ are zero and since $\sum Y_i =1$, we have, summing all lines in the line $j_0$ and all columns in 
the column $j_0$ that
$$
\det(-G^Y+ 2 D_V {}^t \un D_V)= 2 \det( (-G^Y)_{|V\setminus\{j_0\}}),
$$ for any choice of vertex $j_0$.

\section{Appendix}
\subsection{Bases of $\lll_0$, definition~\ref{lem_bases}.}
\label{bases}
Let us first describe the bases of $\lll=\{ x\in \hhh, \; \dive(x)=0\}$.
Consider the graph $\overline \ggg$ where $\overline \ggg=\ggg$ if $V_0=\emptyset$ and $\overline \ggg=(V_1\cup V_1^*\cup\{\delta\}, E)$ where the vertices in $V_0$ are contracted to a single vertex $\delta$ and where $E$ denotes the edges obtained from the edges of $\ggg$ by identification of the vertices of $V_0$ to the single point $\delta$ (note that by a slight abuse of notation we also denote by $E$ the edges of $\overline \ggg$ since each edge of $\ggg$ gives one edge of $\overline \ggg$ ; in $\overline \ggg$ we may of course have multiple edges and also loops $(\delta,\delta)$ obtained by the contraction of $V_0$). 
\begin{lemma}
A subset $B\subset \tilde E$ is a basis of $\lll$ if and only if $E\setminus (B\cup B^*)$ is a spanning tree of $\overline \ggg$ (where we understand
that $B^*=\{(j^*,i^*), \; (i,j)\in B\}$).
\end{lemma}
\begin{proof}
Indeed, we have 
$$
\lll^\perp \cap \hhh =\{\nabla h, \;\; h\in \aaa\}.
$$
Indeed classically, 
$$
\lll^\perp =(\ker\dive\cap \hhh)^\perp= \im \nabla+ \hhh^\perp,
$$ 
hence $\lll^\perp\cap \hhh= \im\nabla\cap \hhh$ and $\nabla h\in \hhh$ if and only if there exists a constant $c$ such that $h-c\in \aaa$.

This implies that if $T\subset \tilde E$ is a basis of $\lll^\perp\cap \hhh$ then $T\cup T^*$ is a spanning tree of $\overline \ggg$. Indeed, $\nabla h(e^*)=\nabla h(e)$
if $h\in \aaa$ and thus $T\cup T^*$ cannot contain any cycle in $\overline \ggg$ (otherwise the corresponding $\nabla h(e)$
would be linearly dependent) and must connect the vertices of $\overline \ggg$.
\end{proof}

Then the bases of $\lll_0$ are obtained from the bases of $\lll$ by $B\setminus \{e_0\}$ where $e_0\in \tilde E\setminus B$. The fact that the Jacobian between
two different bases is $\pm 1$ is a consequence of the computation of the change section.

\subsection{The Jacobian of the change of variable in Lemma \ref{Ortho}}
\label{jac}
The subset  $E\setminus (B\cup B^*)$ is a spanning tree of $\overline \ggg$; moreover it can always be decomposed in the
following way : there exists $\overline T\subset E$ such that
 \begin{itemize}
\item $\overline T$ is a connected tree in $\overline \ggg$.
\item
$\overline T\cap \overline T^*=\emptyset$ and $\overline T\cup \overline T^*=E\setminus (B\cup B^*)$.
\end{itemize}
Let $\overline V_1$ be the set of vertices connected by $\overline T$ where we eventually remove $\delta$. Then $\overline V_1^*$ is the set of
vertices connected by  $\overline T^*$  where we eventually remove $\delta$ and 
$\overline V_1\cap \overline V_1^*=\emptyset$, $\overline V_1\cup \overline V_1= V_1\cup V_1^*$ (i.e. $\overline V_1$ is a 
representation of the points of $V_1\cup V_1^*$ quotiented by the identification of $i$ with $i^*$, possibly different from $V_1$).
It means that up to a change of definition of $\tilde E$, $V_1$ and $T$ in section~\ref{bases} we can always assume that
\begin{itemize}
\item
$T\cup T^*= E\setminus(B\cup B^*)$ and $T\cap T^*=\emptyset$
\item
$T\subset \tilde E$ is a tree connecting the vertices $\{\delta\}\cup V_1$ (or $V_1$ if $V_0=\emptyset$), hence $T^*$ is a tree connecting the vertices $\{\delta\}\cup V_1^*$.
\end{itemize}
Up to a sign, the change of variable in Lemma \ref{Ortho} is the same as the change of variable
$$
(x_e)_{e\in \tilde E} \in \hhh\rightarrow (\lambda,(z_e)_{e\in B},(h_i)_{i\in V_1})\in \R\times\lll_0\times \aaa
$$
We decompose this change of variables in three steps. The first step
$$
(x_e)_{e\in\tilde E}\rightarrow (\lambda, y=x-\lambda Y)\in \R\times \hhh_0
$$
has clearly Jacobian $1$ for any choice of basis $(y_e)_{e\in \tilde E\setminus \{e_0\}}$, for some $e_0\in \tilde E$.
Since, $\dive(Y)=0$ we have $h=\demi\dive(x)=\demi\dive(y)$ and we then make the change of variables
$$
((y_e)_{e\in B\setminus\{e_0\}}, (y_e)_{e\in T}) \rightarrow ((y_e)_{e\in B\setminus\{e_0\}}, (h_i)_{i\in V_1})
$$
which is linear with matrix
$$
\left(\begin{array}{ll}

\Id& 0
\\
*& J_T
\end{array}\right)
$$
where $J_T$ is the incident matrix of the tree $T$ given by
$$
(J_T)_{e,i}=\left\{\begin{array}{ll} 1,& \hbox{ if $\underline e=i$}
\\
-1, &\hbox{ if $\overline e=i$}
\\
0,&\hbox{ otherwise}
\end{array}\right.
$$
for $e\in T$ and $i\in \overline V_1$.
Since $T$ is a spanning tree on $\overline V_1$, then $\det(J_T)= \pm 1$.
In the last step we make the change of variables
$$
((y_e)_{e\in B\setminus\{e_0\}}, (h_i)_{i\in V_1})
\rightarrow
((z_e)_{e\in B\setminus\{e_0\}}, (h_i)_{i\in V_1}).
$$
which has matrix
$$
\left(\begin{array}{ll}
\Id& *
\\
0& \Id
\end{array}\right)
$$
and has Jacobian 1.

Hence, the change of variable of Lemma \ref{Ortho} is $\pm 1$. The fact the the Jacobian does not depend on the choice of the
basis $B\setminus \{e_0\}$ implies that the Jacobian between two of these bases is always $\pm 1$. 

\paragraph{{\bf Acknowledgment.}} 
This work is supported by National Science Foundation of China (NSFC), 
grant No.\ 11771293, by the LABEX MILYON
(ANR-10-LABX-0070) of Universit\'e de Lyon and the project ANR LOCAL (ANR-22-CE40-0012-02) operated by the Agence Nationale de la Recherche (ANR) 
in France, and by the Australian Research Council (ARC) grant  DP180100613.

\footnotesize
\bibliographystyle{plain}
\bibliography{k-dependent-ERRW}

\end{document}